\documentclass[12pt,a4paper]{amsart}
\usepackage[margin=20mm,top=28mm]{geometry}

\usepackage{amsmath,amssymb,amsfonts,amsthm,mathtools,mathrsfs}
\usepackage[colorlinks=true,linkcolor=cyan,citecolor=magenta,urlcolor=blue]{hyperref}

\usepackage[T1]{fontenc}
\usepackage[utf8]{inputenc}
\usepackage{newtxtext,newtxmath}

\usepackage{enumerate}

\newtheorem{theorem}{Theorem}[section]
\newtheorem{proposition}[theorem]{Proposition}
\newtheorem{lemma}[theorem]{Lemma}
\newtheorem{corollary}[theorem]{Corollary}
\theoremstyle{remark}
\newtheorem{remark}[theorem]{Remark}

\newcommand{\Z}{\mathbb{Z}}

\title{Birational Equivalences for Kac--Moody Borel Enveloping Algebras}
\author{Atabey Kaygun}
\address{Istanbul Technical University, Istanbul, Turkey}
\email{kaygun@itu.edu.tr}
\date{}

\begin{document}
\begin{abstract}
A Coxeter ordering of the simple roots of a finite-rank Kac--Moody algebra determines a finite family of commuting real-root vectors.  We prove that $U^{\geq0}(\mathfrak g)$ is birationally equivalent to $Z\otimes\mathbb A_n$, where $Z$ is the residual Coxeter centralizer, by identifying the Coxeter localization $U^{\geq0}(\mathfrak g)[\mathbf X^{-1}]$ with $Z\otimes\mathbb A_n[\mathbf x^{-1}]$.  For symmetrizable Cartan matrices the residual algebra is generated by finite Coxeter windows and is finitely presented.  For the generic quantum Borel with torus dual to the root lattice, we prove the analogous birational equivalence.
\end{abstract}
\maketitle

\section*{Introduction}

The Gelfand--Kirillov problem asks when an enveloping algebra admits Weyl coordinates after passing to a suitable quotient field or localization \cite{GelfandKirillov1966Doklady,GelfandKirillov1966}.  For Borel subalgebras of finite-dimensional semisimple Lie algebras, the quotient division ring is understood by the classical results of Joseph and McConnell \cite{Joseph1974,McConnell1974}.  For Kac--Moody Borels, however, a global quotient division ring is not always the right object: the algebra is generally infinite-dimensional, while the root lattice still suggests a finite collection of distinguished directions that should be separated first.

A Coxeter ordering provides the required finite family of directions.  The associated inversion roots form a basis of the root lattice, and corresponding real-root vectors $\mathbf X=(X_1,\dots,X_n)$ commute.  Let $L=U^{\geq0}(\mathfrak g)[\mathbf X^{-1}]$ and $Z=C_L(\mathfrak h,\mathbf X)$.  Our main classical result is $U^{\geq0}(\mathfrak g)$ and $Z\otimes\mathbb A_n$ are birationally equivalent algebras.  The key structural statement is the explicit localization isomorphism
\[
L\cong Z\otimes\mathbb A_n[\mathbf x^{-1}].
\]
Thus the birational equivalence separates an explicit rank-$n$ Weyl factor, while all remaining birational information is concentrated in $Z$.

The same birational construction also gives a finite description of the residual algebra.  Its degree-zero part is $Z[\mathbf b]$, where the components of $\mathbf b$ are dual to the Coxeter roots.  Normalized simple-root coefficients generate finite windows under the associated difference operators, and the invariant projections of the window elements generate $Z$.  When the Cartan matrix is symmetrizable this gives a finite presentation.  In finite type, if $N=|\Delta_+|$, one obtains $\operatorname{GKdim}Z=N-n$.

There is a parallel generic quantum statement for the positive quantum Borel with torus $P^\vee=\operatorname{Hom}(Q,\Z)$.  We prove $U_q^{\geq0}(\mathfrak g)$ is birationally equivalent to $Z_q\otimes(\mathbb A_1^q)^{\otimes n}$.  The proof identifies the Coxeter localization of $U_q^{\geq0}(\mathfrak g)$ with $Z_q\otimes\mathbb T_q^{\otimes n}$ and then uses that the rank-one quantum torus and quantum Weyl algebra are birationally equivalent.  When $\mathfrak h_0\ne0$, this torus convention quantizes $U(\mathfrak b_+/\mathfrak h_0)$ rather than the full classical Borel; this is made explicit in Section~\ref{sec:quantum-splitting}.

By the Chevalley involution, the classical results have identical counterparts for the negative Borel \(U^{\leq0}(\mathfrak g)=U(\mathfrak b_-)\): the corresponding negative Coxeter root vectors give the same localization and birational-equivalence statements. The analogous observation applies to the generic quantum negative Borel. We therefore formulate the results only for the positive Borel. These statements do not, however, imply corresponding birational decompositions for the full enveloping algebra \(U(\mathfrak g)\), or for the full quantum group, where the interaction between the positive and negative parts introduces additional structure.

The examples are chosen to show what the residual core records.  In finite type, $A_1$ has trivial residual core, $A_2$ has a one-variable polynomial core, and for both Coxeter orderings in $B_2$ the residual core is birationally equivalent to the first Weyl algebra.  In affine type the null-root lattice direction already lies in the explicit Weyl factor, while normalized coefficients coming from imaginary-root spaces occur in the essential residual core.  The cases $A_1^{(1)}$, $A_2^{(1)}$, and the rank-two non-simply-laced affine diagram make this distinction concrete.

\subsection*{Relation to previous work}

Real-root localizations of Kac--Moody and quantum enveloping algebras occur in Iohara--Malikov \cite{IoharaMalikov1994}.  Twisted localization in finite type is used by Mathieu \cite{Mathieu2000}.  The feature specific to the present construction is the Coxeter choice: the inversion roots form a lattice basis and the corresponding root vectors commute.  This identifies the localized Borel with a localized Weyl algebra tensored with the residual centralizer, and hence gives an explicit birational model.

The birational language is compatible with the framework of \cite{Kaygun2025GWA}.  The present result identifies a concrete Weyl factor in the birational type of the Borel and isolates the residual algebra, instead of assuming that the entire Borel has a prescribed Weyl-field form.  The generic quantum statement is obtained from the same Coxeter data by replacing additive translations with toral $q$-translations.

\subsection*{Organization}

Sections~\ref{sec:classical-denominators}--\ref{sec:classical-windows} develop the classical Coxeter birational equivalence, its localization construction, and the finite Coxeter-window presentation of the residual algebra.  Section~\ref{sec:quantum-splitting} proves the generic quantum analogue.  Section~\ref{sec:examples} treats the finite and affine examples $A_1$, $A_2$, and $B_2$.

\subsection*{Notations and conventions}
We work over a field $k$ of characteristic zero; algebras are associative and unital, unadorned tensor products are over $k$, and $\mathbb N=\{0,1,2,\dots\}$.  The Weyl algebra $\mathbb A_n$ is generated by $x_1,\dots,x_n,y_1,\dots,y_n$ with $[y_r,x_s]=\delta_{rs}$ and all other pairs commuting.  We write $S(V)$ for the symmetric algebra of a vector space $V$.

For $x,a$ in an algebra put $\operatorname{ad}(x)(a)=xa-ax$, and $\operatorname{Ad}(x)(a)=xax^{-1}$ when $x$ is invertible.  The centralizer of $S\subseteq R$ is $C_R(S)=\{a\in R:as=sa\text{ for all }s\in S\}$.  If $\mathbf a=(a_1,\dots,a_m)$ is a finite tuple, $C_R(S,\mathbf a)$ denotes the centralizer of $S$ together with all components of $\mathbf a$.  We use the letter $Z$ below for the Coxeter centralizer; it does not denote the ordinary center of an algebra.

An element $x$ of an algebra $R$ is regular if it is not a zero divisor, and normal if $xR=Rx$.  A linear operator on $R$ is locally nilpotent, respectively locally finite, if every element is killed by some power of it, respectively lies in a finite-dimensional invariant subspace; it is locally diagonalizable if it is locally finite and diagonalizable on every finite-dimensional invariant subspace.  A multiplicative set of regular elements is a denominator set if it satisfies the left and right Ore conditions, and $R[\Sigma^{-1}]$ denotes the corresponding localization.  For a tuple $\mathbf x=(x_1,\dots,x_n)$ whose products form a denominator set, $R[\mathbf x^{-1}]$ denotes the localization at these products.  If $R$ is graded by an abelian group, $R_{(0)}$ denotes its degree-zero part.  If a group $\Gamma$ acts on $R$ by automorphisms, $R\#k[\Gamma]$ denotes the corresponding skew group algebra.  We write $\operatorname{GKdim}$ for Gelfand--Kirillov dimension \cite{KrauseLenagan2000}.  For a central element $b$ and $m\in\mathbb N$ put $\binom bm=b(b-1)\cdots(b-m+1)/m!$.

For Kac--Moody algebras we follow \cite{Kac1990}.  The root and positive-root sets are $\Delta$ and $\Delta_+$, the root space of $\gamma$ is $\mathfrak g_\gamma$, and $Q=\bigoplus_i\mathbb Z\alpha_i$ is the root lattice.  We write $U^{\geq0}(\mathfrak g)=U(\mathfrak b_+)$ for the positive Borel enveloping algebra.

A morphism $\varphi:R\to R'$ of domains is a \emph{birational equivalence} if there are denominator sets $\Sigma\subseteq R$ and $\Sigma'\subseteq R'$ with $\varphi(\Sigma)\subseteq\Sigma'$ such that the induced map $R[\Sigma^{-1}]\to R'[\Sigma'^{-1}]$ is an isomorphism.  We write $R\sim_{\rm bir}R'$, and say that $R$ and $R'$ are birationally equivalent, when they are connected by a finite zigzag of such birational equivalences.

\section{Classical Coxeter denominators}
\label{sec:classical-denominators}

Let $I=\{1,\dots,n\}$ and fix a minimal realization of $C=(a_{ij})$.  Choose an ordering $I=\{i_1,\dots,i_n\}$ and put
$\beta_r=s_{i_1}\cdots s_{i_{r-1}}(\alpha_{i_r})$.

\begin{lemma}\label{lem:coxeter-roots}
The roots $\beta_1,\dots,\beta_n$ form a $\mathbb Z$-basis of $Q$, and $\beta_r+\beta_s\notin\Delta$ for $r\ne s$.
\end{lemma}

\begin{proof}
Each reflection preceding $\alpha_{i_r}$ can only add an integral multiple of an earlier simple root, so the transition matrix from the ordered simple roots to the $\beta_r$ is integral unitriangular.  For $r<s$, put $\gamma=s_{i_{r+1}}\cdots s_{i_{s-1}}(\alpha_{i_s})$.  The root $\gamma$ is positive real and has zero $\alpha_{i_r}$-coefficient, hence $\gamma-\alpha_{i_r}$ is not a root.  The real-root string through $\gamma$ in the $\alpha_{i_r}$-direction therefore has length $m=-\gamma(h_{i_r})$ in the positive direction, so $\gamma+(m+1)\alpha_{i_r}$ is not a root; compare \cite[Chapter~5, Proposition~5.1]{Kac1990}.  If $w_{r-1}=s_{i_1}\cdots s_{i_{r-1}}$, then $\beta_r=w_{r-1}(\alpha_{i_r})$ and $\beta_s=w_{r-1}s_{i_r}(\gamma)=w_{r-1}(\gamma+m\alpha_{i_r})$.  Hence $\beta_r+\beta_s=w_{r-1}(\gamma+(m+1)\alpha_{i_r})$ is not a root.
\end{proof}

Choose nonzero $X_r\in\mathfrak g_{\beta_r}$ and put $\mathbf X=(X_1,\dots,X_n)$.  Lemma~\ref{lem:coxeter-roots} gives $[X_r,X_s]=0$ for $r\ne s$.  Let $\Sigma_{\mathbf X}$ be the monoid generated by the components of $\mathbf X$.

\begin{proposition}\label{prop:coxeter-ore}
The monoid $\Sigma_{\mathbf X}$ is a two-sided denominator set in $U^{\geq0}(\mathfrak g)$.
\end{proposition}

\begin{proof}
The PBW filtration has associated graded algebra $S(\mathfrak b_+)$, so $U^{\geq0}(\mathfrak g)$ is a domain.  For every real-root vector $X_r$, $\operatorname{ad}X_r$ is locally nilpotent on $\mathfrak g$: it is nilpotent on $\mathfrak h$, and the real-root string in the $\beta_r$-direction is finite on each root space.  Hence it is locally nilpotent on $U^{\geq0}(\mathfrak g)$ by the derivation rule.

If $x$ is nonzero in a domain and $\delta=\operatorname{ad}x$ is locally nilpotent, then
\[
x^Na=\sum_{k=0}^N\binom Nk\delta^k(a)x^{N-k},\qquad
ax^N=\sum_{k=0}^N(-1)^k\binom Nkx^{N-k}\delta^k(a).
\]
Given $a$ and $m\in\mathbb N$ with $\delta^{K+1}(a)=0$, taking $N=m+K$ gives $x^Na\in Rx^m$ and $ax^N\in x^mR$, which are the two Ore conditions for the powers of $x$.  Apply this successively to the commuting $X_r$.  Since $\operatorname{ad}X_r$ annihilates the inverses of the other $X_s$, local nilpotence persists after each inversion; in particular the localization may be obtained by inverting the $X_r$ successively in any order.
\end{proof}

\section{Classical Coxeter birational equivalence}
\label{sec:classical-splitting}

Put $L=U^{\geq0}(\mathfrak g)[\mathbf X^{-1}]$ and $R=L_{(0)}$.
If $\lambda=\sum_rm_r\beta_r\in Q$, put $X^\lambda=X_1^{m_1}\cdots X_n^{m_n}$.

\begin{proposition}\label{prop:coxeter-lattice}
Every homogeneous element of $L_\lambda$ has a unique expression $aX^\lambda$ with $a\in R$.  Consequently
$L\cong R\#k[Q]$,
where $\lambda\in Q$ acts on $R$ by $\sigma_\lambda=\operatorname{Ad}(X^\lambda)$.
\end{proposition}

\begin{proof}
The $\beta_r$ form a basis of $Q$, and the commuting $X_r$ are units in $L$.  Thus $u\in L_\lambda$ satisfies $u(X^\lambda)^{-1}\in R$, while $X^\lambda X^\mu=X^{\lambda+\mu}$.
\end{proof}

For each $r$ put $\sigma_r=\operatorname{Ad}(X_r)$ and $D_r=\sigma_r-1$.

\begin{proposition}\label{prop:coxeter-differences}
The $D_r$ commute, preserve $R$, and are locally nilpotent on $L$.  For $a\in U^{\geq0}(\mathfrak g)$,
$D_r^k(a)=(\operatorname{ad}X_r)^k(a)X_r^{-k}$.
\end{proposition}

\begin{proof}
The conjugations commute because the $X_r$ commute.  The stated identity follows by induction.  Proposition~\ref{prop:coxeter-ore} gives local nilpotence before localization, and adjoining the inverses does not change it because every $X_s^{\pm1}$ is fixed by every $\sigma_r$.
\end{proof}

Choose $\mathbf b=(b_1,\dots,b_n)\in\mathfrak h^n$ with $\beta_r(b_s)=\delta_{rs}$.  Every degree-zero element commutes with $\mathfrak h$, hence the components of $\mathbf b$ are central in $R$, and $\sigma_r(b_s)=b_s-\delta_{rs}$.

\begin{lemma}\label{lem:difference-slice}
Let $S$ be a $k$-algebra, $\sigma\in\operatorname{Aut}(S)$, and $D=\sigma-1$.  If $D$ is locally nilpotent and a central element $b\in S$ satisfies $\sigma(b)=b-1$, then $S$ is the polynomial ring $S^\sigma[b]$.
\end{lemma}

\begin{proof}
Put $p_m(b)=\binom{-b}{m}$.  Then $Dp_m=p_{m-1}$ and $D^mp_m=1$.  If $m$ is maximal with $D^m(a)\ne0$, then $D^m(a)\in S^\sigma$; subtracting $p_m(b)D^m(a)$ lowers the $D$-degree.  Induction gives existence, and applying the highest power of $D$ gives uniqueness.
\end{proof}

Define the Coxeter centralizer $Z=C_L(\mathfrak h,\mathbf X)$.
Equivalently, $Z$ is the simultaneous $\sigma_r$-fixed subalgebra of $R$.

\begin{proposition}\label{prop:coefficient-polynomial}
Multiplication gives an algebra isomorphism
$R\cong Z[\mathbf b]$.
The root lattice acts trivially on $Z$ and by integral translations on $\mathbf b$.  Consequently
$L\cong Z[\mathbf b]\#k[Q]$.
\end{proposition}

\begin{proof}
Apply Lemma~\ref{lem:difference-slice} successively to $(R,\sigma_r,b_r)$.  Since the conjugations commute, each successive fixed algebra is stable under the remaining automorphisms.  The final common fixed algebra is $Z$.
\end{proof}

\begin{theorem}\label{thm:coxeter-weyl}
There is an algebra isomorphism
\[
U^{\geq0}(\mathfrak g)[\mathbf X^{-1}]
\cong Z\otimes\mathbb A_n[\mathbf x^{-1}].
\]
\end{theorem}

\begin{proof}
The preceding propositions give normal forms $a\,b_1^{k_1}\cdots b_n^{k_n}X^\lambda$ with $a\in Z$.  Put $x_r=X_r$ and $y_r=X_r^{-1}b_r$.  Then $[y_r,x_s]=\delta_{rs}$ and all other Weyl commutators vanish, while $Z$ commutes with every $x_r,y_r$.  Since $x_ry_r=b_r$, the basis $(x_1y_1)^{k_1}\cdots(x_ny_n)^{k_n}x^\lambda$ of $\mathbb A_n[\mathbf x^{-1}]$ is sent to $b_1^{k_1}\cdots b_n^{k_n}X^\lambda$, and the normal forms give injectivity and surjectivity of the tensor-product homomorphism.
\end{proof}

The theorem embeds $Z\otimes\mathbb A_n$ in the domain $U^{\geq0}(\mathfrak g)[\mathbf X^{-1}]$, so $Z\otimes\mathbb A_n$ is a domain.  Each $\operatorname{ad}x_r$ is locally nilpotent on it, so the powers of the $x_r$ form a denominator set whose localization is $Z\otimes\mathbb A_n[\mathbf x^{-1}]$.  Thus $U^{\geq0}(\mathfrak g)$ and $Z\otimes\mathbb A_n$ admit isomorphic denominator localizations; hence
\[
U^{\geq0}(\mathfrak g)\sim_{\rm bir}Z\otimes\mathbb A_n.
\]

Put $\mathfrak h_0=\bigcap_{\lambda\in Q}\ker\lambda$ and define the essential Coxeter core by $Z_{\rm ess}=Z/(\mathfrak h_0)$, where $(\mathfrak h_0)$ denotes the two-sided ideal generated by the central subspace $\mathfrak h_0$.

\begin{corollary}\label{cor:corank-separation}
After choosing a complement of $\mathfrak h_0$ in $\mathfrak h$, there is a tensor decomposition
$Z\cong S(\mathfrak h_0)\otimes Z_{\rm ess}$.
\end{corollary}

\begin{proof}
Choose a vector-space complement $\mathfrak h=\mathfrak h_0\oplus\mathfrak h_1$.  Since $\mathfrak h_0$ is central and $[\mathfrak h_1,\mathfrak n_+]\subseteq\mathfrak n_+$, one has a direct sum of Lie algebras $\mathfrak b_+=\mathfrak h_0\oplus(\mathfrak h_1\ltimes\mathfrak n_+)$.  Hence PBW gives $U^{\geq0}(\mathfrak g)\cong S(\mathfrak h_0)\otimes U(\mathfrak h_1\ltimes\mathfrak n_+)$.  The Coxeter root vectors lie in the second tensor factor, so both the Coxeter localization and its centralizer split off $S(\mathfrak h_0)$.  Quotienting by the ideal generated by $\mathfrak h_0$ identifies the second factor with $Z_{\rm ess}$.
\end{proof}

Consequently $U^{\geq0}(\mathfrak g)\sim_{\rm bir}S(\mathfrak h_0)\otimes Z_{\rm ess}\otimes\mathbb A_n$.

\section{Coxeter windows and finite presentation}
\label{sec:classical-windows}

Proposition~\ref{prop:coefficient-polynomial} identifies $R$ with $Z[\mathbf b]$.  Evaluation at $\mathbf b=0$ defines a projection $\operatorname{ev}_0:R\to Z$.  For $1\le r\le n$, Lemma~\ref{lem:difference-slice} identifies $R$ with the polynomial ring $R^{\sigma_r}[b_r]$, and $\operatorname{ev}_{b_r=0}:R\to R^{\sigma_r}$ denotes evaluation at $b_r=0$.

\begin{lemma}\label{lem:newton-projection}
For $a\in R$ and $1\le r\le n$,
$\operatorname{ev}_{b_r=0}(a)=\sum_{k\ge0}\binom{b_r}{k}D_r^k(a)$,
with a finite sum.
\end{lemma}

\begin{proof}
For every integer $m$, $\sum_{k\ge0}\binom mkD_r^k(a)=\sigma_r^m(a)$.  Both sides are polynomial in $m$ because $D_r$ is locally nilpotent.  For $a=p(b_r)c$ with $c$ fixed by $\sigma_r$, the right side is $p(b_r-m)c$; substituting the central element $b_r$ for $m$ gives $p(0)c$.
\end{proof}

\begin{lemma}\label{lem:newton-expansion}
For $a\in R$,
\[
a=\sum_{\boldsymbol\ell\in\mathbb N^n}
\left(\prod_{r=1}^n\binom{-b_r}{\ell_r}\right)
\operatorname{ev}_0(D^{\boldsymbol\ell}a),
\]
with only finitely many nonzero terms.
\end{lemma}

\begin{proof}
This is the iterated Newton expansion obtained from Lemma~\ref{lem:difference-slice}; the commuting difference operators also commute with evaluation in the other variables.
\end{proof}

Since $\beta_r(b_s)=\delta_{rs}$ and $\dim\mathfrak h=n+\dim\mathfrak h_0$ for a minimal realization, $\mathfrak h=\mathfrak h_0\oplus\bigoplus_rkb_r$.  If $\alpha_i=\sum_rm_{ri}\beta_r$, define
$a_i=e_i(X^{\alpha_i})^{-1}\in R$.
For $\mathbf k\in\mathbb N^n$ put $z_{i,\mathbf k}=D^{\mathbf k}(a_i)$ and
$\mathcal W_i=\{\mathbf k:z_{i,\mathbf k}\ne0\},\qquad c_{i,\mathbf k}=\operatorname{ev}_0(z_{i,\mathbf k})$.

\begin{lemma}\label{lem:finite-windows}
Every window $\mathcal W_i$ is finite.
\end{lemma}

\begin{proof}
Each $D_r$ is locally nilpotent on $a_i$, and the $D_r$ commute.  Bounding the exponent in each coordinate gives a finite box containing the window.
\end{proof}

\begin{proposition}\label{prop:window-generation}
The Coxeter core $Z$ is generated by $\mathfrak h_0$ and the elements $c_{i,\mathbf k}$ with $\mathbf k\in\mathcal W_i$.
\end{proposition}

\begin{proof}
One has
\[
z_{i,\mathbf k}=(\operatorname{ad}X_1)^{k_1}\cdots(\operatorname{ad}X_n)^{k_n}(e_i)
\cdot X^{-\deg((\operatorname{ad}\mathbf X)^{\mathbf k}(e_i))}.
\]
The Coxeter localization $L$ is generated by $\mathfrak h$, the coefficients $a_i$, and the homogeneous units $X^\lambda$.  Moving the units to the right shows that $R$ is generated by $\mathfrak h$ and the $Q$-orbits of the $a_i$.  If $\lambda=\sum_rm_r\beta_r$, then
$\sigma_\lambda(a)=\prod_r(1+D_r)^{m_r}(a)$,
and the generalized binomial expansion is finite because the $D_r$ are locally nilpotent.  Hence $R$ is generated by $\mathfrak h_0$, the $b_r$, and the window elements.  Setting $\mathbf b=0$ proves the claim.
\end{proof}

The Newton expansion reconstructs every window element from its invariant coefficients:
\[
z_{i,\mathbf k}=\sum_{\boldsymbol\ell\in\mathbb N^n}
\left(\prod_r\binom{-b_r}{\ell_r}\right)c_{i,\mathbf k+\boldsymbol\ell},
\]
where coefficients outside $\mathcal W_i$ are zero.

\begin{theorem}\label{thm:finite-presentation}
If $C$ is symmetrizable, then $Z$ is finitely presented.
\end{theorem}

\begin{proof}
By the Gabber--Kac presentation, $U^{\geq0}(\mathfrak g)$ is finitely presented by $\mathfrak h$, the $e_i$, the Cartan relations, and the Serre relations; see \cite[Theorem~9.11]{Kac1990}.  Adjoining inverse symbols for the $X_r$ gives a finite presentation of $L$ by Proposition~\ref{prop:coxeter-ore}.

Let $F$ be the free algebra on the finite window generating set of Proposition~\ref{prop:window-generation}, together with a basis of $\mathfrak h_0$, and let $K$ be the kernel of $F\to Z$.  By Theorem~\ref{thm:coxeter-weyl},
$F\otimes\mathbb A_n[\mathbf x^{-1}]\longrightarrow L$
is surjective with kernel $K\otimes\mathbb A_n[\mathbf x^{-1}]$.  Source and target are finitely presented, hence this kernel is finitely generated as a two-sided ideal: for a surjection between finitely presented algebras, one may use the images of a fixed finite generating set for the source and apply the usual Tietze argument to the two finite presentations.  Expand a finite set of generators of $K\otimes\mathbb A_n[\mathbf x^{-1}]$ in a fixed vector-space basis of $\mathbb A_n[\mathbf x^{-1}]$, and let $K_0\subseteq F$ be the two-sided ideal generated by the finitely many coefficients that occur.  Then the tensor ideal generated by these elements is contained in $K_0\otimes\mathbb A_n[\mathbf x^{-1}]$ but equals $K\otimes\mathbb A_n[\mathbf x^{-1}]$.  Hence $k\otimes1\in K_0\otimes\mathbb A_n[\mathbf x^{-1}]$ for every $k\in K$, so $K\subseteq K_0$.  The reverse inclusion is immediate, and therefore $K=K_0$ is finitely generated.
\end{proof}

\begin{corollary}\label{cor:gkdim}
Assume $C$ is of finite type and put $N=|\Delta_+|$.  Then
$\operatorname{GKdim}Z=N-n$.
\end{corollary}

\begin{proof}
The PBW filtration gives $\operatorname{GKdim}U^{\geq0}(\mathfrak g)=N+n$.  Theorem~\ref{thm:coxeter-weyl} identifies the Coxeter localization with $Z\otimes\mathbb A_n[\mathbf x^{-1}]$, and localization at the Coxeter root vectors preserves Gelfand--Kirillov dimension \cite[Theorem~4.9]{KrauseLenagan2000}.  The localized Weyl factor has Gelfand--Kirillov dimension $2n$ and limit-type growth, so tensoring with it adds $2n$ \cite[Chapter~3]{KrauseLenagan2000}.  Apply Theorem~\ref{thm:coxeter-weyl}.
\end{proof}

\section{Quantum Coxeter birational equivalence}
\label{sec:quantum-splitting}

Assume that $C$ is symmetrizable and work over the field $k(q)$, with $q$ indeterminate.  Let $P^\vee=\operatorname{Hom}(Q,\Z)$, put $\mathcal K=\{K_h:h\in P^\vee\}$, and let $U_q^{\geq0}(\mathfrak g)$ denote the positive Drinfeld--Jimbo Borel with this torus, generated by the $E_i$ and the $K_h$ subject to the quantum Serre relations, $K_0=1$, $K_hK_{h'}=K_{h+h'}$ and $K_hE_iK_h^{-1}=q^{\alpha_i(h)}E_i$; let $U_q^+$ be its subalgebra generated by the $E_i$.  Both algebras are domains.  For $U_q^+$, Lusztig's integral form is a free $\Z[q,q^{-1}]$-module whose specialization at $q=1$ is $U(\mathfrak n_+)$ \cite[Chapter~14 and Section~33.1]{Lusztig1993}; a relation $ab=0$ with $a,b\ne0$ would, after clearing denominators and dividing out the highest powers of $q-1$, give one in the domain $U(\mathfrak n_+)$.  Moreover $U_q^{\geq0}(\mathfrak g)$ is a skew group algebra of $U_q^+$ over the torsion-free group $P^\vee$.  This choice of torus ensures that the Coxeter basis has integral dual elements.  Since $P^\vee$ has rank $n$, the torus quantizes $\mathfrak h/\mathfrak h_0$ rather than $\mathfrak h$; thus $U_q^{\geq0}(\mathfrak g)$ is a quantization of $U(\mathfrak b_+/\mathfrak h_0)$, and the residual centralizer $Z_q$ below is the analogue of $Z_{\rm ess}$ in Corollary~\ref{cor:corank-separation}.

For the same Coxeter ordering set $X_r=T_{i_1}\cdots T_{i_{r-1}}(E_{i_r})$,
where the $T_i$ are Lusztig's braid automorphisms \cite[Chapter~37]{Lusztig1993}, normalized as in \cite{BerensteinGreenstein2017}.  Since $s_{i_1}\cdots s_{i_r}$ is reduced, each $X_r$ lies in $U_q^+$.  Again write $\mathbf X=(X_1,\dots,X_n)$, and $U_q^{\geq0}(\mathfrak g)[\mathbf X^{-1}]_{(0)}$ for the degree-zero part of $U_q^{\geq0}(\mathfrak g)[\mathbf X^{-1}]$.
The quantum Schubert cell of a repetition-free Weyl element is a quantum affine space \cite[Lemma~5.1]{BerensteinGreenstein2017}.  Hence $X_rX_s=q^{\omega_{rs}}X_sX_r$ for integers $\omega_{rs}=-\omega_{sr}$.

Define $Z_q=C_{U_q^{\geq0}(\mathfrak g)[\mathbf X^{-1}]}(\mathcal K,\mathbf X)$.

\begin{proposition}\label{prop:quantum-ore}
The multiplicative monoid generated by $\mathbf X$ is a two-sided denominator set in $U_q^{\geq0}(\mathfrak g)$.
\end{proposition}

\begin{proof}
For a simple generator $E_i$, the simple-root localization of the full quantum group is well defined; see \cite[Section~3.1]{IoharaMalikov1994}.  Consider $\operatorname{Ad}(E_i)$ on this localization.  If $j\ne i$ and $N=1-a_{ij}$, the quantum Serre relation gives
\[
\prod_{\ell=0}^{N-1}\bigl(\operatorname{Ad}(E_i)-q_i^{N-1-2\ell}\bigr)(E_j)=0,
\qquad q_i=q^{\kappa_i},
\]
where $\kappa_i>0$ and $\kappa_i a_{ij}=\kappa_j a_{ji}$.  The roots of this polynomial are distinct, so every $E_j$ lies in a finite-dimensional semisimple $\operatorname{Ad}(E_i)$-module.  The toral generators are eigenvectors, while $F_j$ is fixed for $j\ne i$.  For $F_i$, writing $K_i=K_{\kappa_ih_i}$, one has
$\operatorname{Ad}(E_i)(F_i)=F_i+(q_i-q_i^{-1})^{-1}(K_i-K_i^{-1})E_i^{-1}$ and $\operatorname{Ad}(E_i)(K_i^{\pm1}E_i^{-1})=q_i^{\mp2}K_i^{\pm1}E_i^{-1}$.  Thus the space spanned by $F_i$, $K_iE_i^{-1}$, and $K_i^{-1}E_i^{-1}$ is stable; in this basis the diagonal eigenvalues are $1,q_i^{-2},q_i^2$, hence are distinct.  Thus every standard generator of the full quantum group lies in a finite-dimensional semisimple $\operatorname{Ad}(E_i)$-stable subspace.  Since products of eigenvectors are eigenvectors, $\operatorname{Ad}(E_i)$ is locally diagonalizable on the simple-root localization, with spectrum in $q^\Z$.  Lusztig braid automorphisms transport this statement to every real Coxeter root vector $X_r$.

Consequently, for every $a\in U_q^{\geq0}(\mathfrak g)$ there is a polynomial $f(t)=c_0t^N+\cdots+c_N$ with $c_0c_N\ne0$ such that $f(\operatorname{Ad}(X_r))(a)=0$.  Clearing denominators gives a finite-adjoint-action relation of the form $\sum_{j=0}^N c_jX_r^{N-j}aX_r^j=0$.  By \cite[Proposition~3.2]{Panov2001}, the powers of each $X_r$ satisfy both Ore conditions.  Because the $X_r$ quasi-commute, this property survives the successive inversions, and their multiplicative monoid is a two-sided denominator set.
\end{proof}

Let $\mathbb T_q=k(q)\langle x^{\pm1},h^{\pm1}\mid hx=qxh\rangle$ be the rank-one quantum torus, and let $\mathbb A_1^q=k(q)\langle x,y\mid yx-qxy=1\rangle$ be the rank-one quantum Weyl algebra.

\begin{theorem}\label{thm:quantum-coxeter-weyl}
There is an algebra isomorphism
\[
U_q^{\geq0}(\mathfrak g)[\mathbf X^{-1}]
\cong
Z_q\otimes\mathbb T_q^{\otimes n}.
\]
\end{theorem}

\begin{proof}
The Coxeter localization exists by Proposition~\ref{prop:quantum-ore}, and its proof shows that each $\operatorname{Ad}(X_r)$ is locally diagonalizable on $U_q^{\geq0}(\mathfrak g)[\mathbf X^{-1}]$ with spectrum in $q^\Z$.  Every degree in $Q$ contains an invertible ordered Laurent monomial in the $X_r$.  Conjugation by the $X_r$ preserves degree, and the conjugations commute because the $X_r$ quasi-commute.  On
$U_q^{\geq0}(\mathfrak g)[\mathbf X^{-1}]_{(0)}$
they are locally simultaneously diagonalizable with eigenvalues in $q^\Z$.

Choose $\mathbf b=(b_1,\dots,b_n)\in(P^\vee)^n$ with $\beta_r(b_s)=\delta_{rs}$.  Put $H_s=K_{b_s}$ and $\mathbf H=(H_1,\dots,H_n)$.  The components of $\mathbf H$ are central in the degree-zero algebra, and $X_rH_sX_r^{-1}=q^{-\delta_{rs}}H_s$.
For $\mathbf m=(m_1,\dots,m_n)\in\Z^n$ write $H^{\mathbf m}=H_1^{m_1}\cdots H_n^{m_n}$.  If a simultaneous eigenvector $a$ satisfies $X_raX_r^{-1}=q^{-m_r}a$ for every $r$, then $aH^{-\mathbf m}\in Z_q$.  Conversely, the joint eigenspace indexed by $\mathbf m$ is $Z_qH^{\mathbf m}$.  Local simultaneous diagonalizability gives a finite joint-eigenspace decomposition of every element, and distinct exponent vectors give distinct characters because $q$ is indeterminate.  Hence
\[
U_q^{\geq0}(\mathfrak g)[\mathbf X^{-1}]_{(0)} \cong Z_q\otimes k(q)[\mathbf H^{\pm1}].
\]
It follows that $U_q^{\geq0}(\mathfrak g)[\mathbf X^{-1}]$ is the tensor product of $Z_q$ with the quantum torus generated by the components of $\mathbf X^{\pm1}$ and $\mathbf H^{\pm1}$, with
\[
X_rX_s=q^{\omega_{rs}}X_sX_r,
\qquad
H_rH_s=H_sH_r,
\qquad
H_rX_s=q^{\delta_{rs}}X_sH_r.
\]

Choose integers $a_{rs}$ with $a_{sr}-a_{rs}=\omega_{rs}$ and replace $X_r$ by $X_rH_1^{a_{r1}}\cdots H_n^{a_{rn}}$.
These new generators commute with one another and satisfy the standard rank-one quantum-torus relations with the $H_s$.  This gives the stated tensor product of $n$ copies of $\mathbb T_q$.
\end{proof}

It remains to relate the quantum torus to the quantum Weyl algebra.  Put $h=1+(q-1)xy$ in $\mathbb A_1^q$.  Then $hx=qxh$ and $yh=qhy$.  The algebra $\mathbb A_1^q$ is an Ore extension of $k(q)[x]$, hence a domain, and the powers of $x$ form a denominator set: induction gives $yx^m=q^mx^my+[m]_qx^{m-1}$ with $[m]_q=(q^m-1)/(q-1)$, so $y^jx^m\in x^{m-j}\mathbb A_1^q$ and $x^my^j\in\mathbb A_1^qx^{m-j}$ for $m\ge j$.  In $\mathbb A_1^q[x^{-1}]$ one has $y=(q-1)^{-1}x^{-1}(h-1)$, so $\mathbb A_1^q[x^{-1}]$ is generated by $x^{\pm1}$ and $h$ subject only to $hx=qxh$; the monomials $h^ax^b$ with $a\in\mathbb N$ and $b\in\Z$ are linearly independent, as their leading terms $x^{a+b}y^a$ show.  The element $h$ is normal there, and $\mathbb A_1^q[x^{-1}][h^{-1}]\cong\mathbb T_q$.  Hence $Z_q\otimes(\mathbb A_1^q)^{\otimes n}$ embeds in the domain $Z_q\otimes\mathbb T_q^{\otimes n}\cong U_q^{\geq0}(\mathfrak g)[\mathbf X^{-1}]$, and localizing each quantum Weyl factor at $x$ and then $h$ gives the same quantum-torus algebra.  Therefore $U_q^{\geq0}(\mathfrak g)$ and $Z_q\otimes(\mathbb A_1^q)^{\otimes n}$ are connected through isomorphic denominator localizations, so
\[
U_q^{\geq0}(\mathfrak g)\sim_{\rm bir}Z_q\otimes(\mathbb A_1^q)^{\otimes n}.
\]

\paragraph{Comparison with the classical coordinates.}
The quantum proof replaces the additive relations $\sigma_r(b_s)=b_s-\delta_{rs}$ by $X_rK_{b_s}X_r^{-1}=q^{-\delta_{rs}}K_{b_s}$.  Formally as $q\to1$, the toral coordinates therefore recover the classical translation coordinates.  Because the chosen quantum torus has rank $n$, it corresponds to the quotient Cartan $\mathfrak h/\mathfrak h_0$, so the residual object $Z_q$ should be compared with $Z_{\rm ess}$ rather than with $Z$.  We do not assert here that an integral form of $Z_q$ specializes to $Z_{\rm ess}$.

\section{Examples}
\label{sec:examples}

The examples below use the same notation as the general construction.  In affine type we use the standard loop realization and write $\delta$ for the primitive null root.  The point of the affine examples is that the homogeneous unit $X^\delta$ belongs to the split Weyl factor, whereas normalized imaginary-root vectors contribute to the essential core $Z_{\rm ess}$.

\subsection{Finite type}

\subsubsection{Type $A_1$}

Here $\beta_1=\alpha_1$, so we take $X_1=e_1$.  The normalized coefficient of the only simple generator is $e_1X_1^{-1}=1$, and $\mathfrak h_0=0$.  Hence $Z=k$, and the localization isomorphism $U^{\geq0}(\mathfrak{sl}_2)[e_1^{-1}]\cong\mathbb A_1[x_1^{-1}]$ exhibits $U^{\geq0}(\mathfrak{sl}_2)\sim_{\rm bir}\mathbb A_1$.  Thus in type $A_1$ the residual core is trivial and the Coxeter localization accounts for the entire birational type of the Borel.

\subsubsection{Type $A_2$}

Take the order $(1,2)$ and put $X_1=e_1$ and $X_2=e_{12}=[e_1,e_2]$.
Then $\beta_1=\alpha_1$ and $\beta_2=\alpha_1+\alpha_2$, so $\alpha_2=\beta_2-\beta_1$.  The normalized coefficient of $e_1=X_1$ is $1$, and the normalized coefficient $a=e_2X_1X_2^{-1}$ satisfies $D_1(a)=1$ and $D_2(a)=0$.  Therefore $c=\operatorname{ev}_0(a)=a+b_1$ generates the residual core.  Since Corollary~\ref{cor:gkdim} gives $\operatorname{GKdim}Z=1$, the element $c$ is transcendental over $k$, so $Z=k[c]$.  Consequently $U^{\geq0}(\mathfrak{sl}_3)\sim_{\rm bir}k[c]\otimes\mathbb A_2$.

\subsubsection{Type $B_2=C_2$}

Take $\alpha_1$ short and $\alpha_2$ long, and write $e_{12}=[e_1,e_2]$ and $e_{112}=[e_1,e_{12}]$.

\begin{proposition}\label{prop:B2-21}
For the order $(2,1)$, the Coxeter core is generated by $p,c$ and these generators satisfy $[p,c]=2p^2-p$.
\end{proposition}

\begin{proof}
Take $X_1=e_2$ and $X_2=e_{12}$; the normalized coefficient of $e_2$ is $1$.  In $\mathfrak n_+$ the only nonzero brackets among $e_1,e_2,e_{12},e_{112}$ are $[e_1,e_2]=e_{12}$ and $[e_1,e_{12}]=e_{112}$.  The element $p=e_{112}X_1X_2^{-2}$ belongs to $Z$.  The window of the normalized coefficient $a_1=e_1X_1X_2^{-1}$ is $\{00,10,01\}$; the coefficient at $10$ is $-1$ and the one at $01$ is $-p$.  Let $c=\operatorname{ev}_0(a_1)$; Lemma~\ref{lem:newton-projection} gives $c=a_1-b_1-pb_2$.  Proposition~\ref{prop:window-generation} gives generation by $p$ and $c$.  Since $p$ commutes with $\mathbf b$, $X_1$ and $X_2$, one has $[p,c]=[p,e_1]X_1X_2^{-1}$.  From $[X_2^{-1},e_1]=e_{112}X_2^{-2}$ one gets $[X_2^{-2},e_1]=2e_{112}X_2^{-3}$, hence
\[
[p,e_1]=e_{112}\bigl(2X_1e_{112}X_2^{-3}-X_2^{-1}\bigr),
\]
and multiplying by $X_1X_2^{-1}$ gives $[p,c]=2p^2-p$.
\end{proof}

\begin{proposition}\label{prop:B2-12}
For the order $(1,2)$, the Coxeter core is generated by two projected coefficients $c_0,c_1$ satisfying $[c_1,c_0]=c_1^2-c_1-2c_0$.  Equivalently, for $c=c_0+\tfrac12(c_1-c_1^2)$ one has $[c_1,c]=-2c$.
\end{proposition}

\begin{proof}
Take $X_1=e_1$ and $X_2=e_{112}$; the normalized coefficient of $e_1$ is $1$, and $e_{112}$ is central in $U(\mathfrak n_+)$.  Put $a=e_2X_1^2X_2^{-1}$ and $z=D_1(a)=e_{12}X_1X_2^{-1}$.  Then $D_1(z)=1$ and $D_2(a)=D_2(z)=0$, so the window of $a$ is $\{00,10,20\}$ with constant coefficient at $20$.  Lemma~\ref{lem:newton-projection} gives $c_1=\operatorname{ev}_0(z)=z+b_1$ and $c_0=\operatorname{ev}_0(a)=a+b_1z+\binom{b_1}{2}$, and Proposition~\ref{prop:window-generation} gives generation.  Since $\mathbf b$ is central in $R$, the stated relation is equivalent to $[z,a]=z^2-z-2a$.  With $w=X_2^{-1}$, one has $z=e_{12}e_1w$ and $a=e_2e_1^2w$, and the relations of $\mathfrak n_+$ give
\[
[e_{12}e_1,e_2e_1^2]=e_{12}^2e_1^2-2e_{112}e_2e_1^2,
\qquad
(e_{12}e_1)^2=e_{12}^2e_1^2+e_{112}e_{12}e_1.
\]
Multiplying by $w^2$ gives $[z,a]=z^2-z-2a$.  The equivalent form follows by substitution.
\end{proof}

\begin{corollary}\label{cor:B2-ordering}
For either Coxeter ordering in type $B_2$, the residual core $Z$ satisfies $Z\sim_{\rm bir}\mathbb A_1$.
\end{corollary}

\begin{proof}
For the order $(2,1)$ put $f=p(2p-1)$.  Since $Z$ is generated by $p,c$ and $[p,c]=f$, the derivation $\operatorname{ad}(f)$ is locally nilpotent on $Z$, so the powers of $f$ form a denominator set.  After inverting $f$, the element $\partial=-f^{-1}c$ satisfies $[\partial,p]=1$, and $p,c$ are recovered from $p,\partial$.  Write $x,y$ for the generators of $\mathbb A_1$.  The same locally nilpotent-adjoint argument shows that the powers of $f(x)$ form a denominator set in $\mathbb A_1$.  Thus $x\mapsto p$, $y\mapsto\partial$ defines a surjective homomorphism from $\mathbb A_1[f(x)^{-1}]$ onto $Z[f^{-1}]$.  A denominator localization of the simple Weyl algebra is simple, so this homomorphism is injective.

For the order $(1,2)$ put $u=-c_1/2$.  Then $[u,c]=c$, so $\operatorname{ad}(c)$ is locally nilpotent on the algebra generated by $c,u$ and the powers of $c$ form a denominator set.  The powers of $x$ form a denominator set in $\mathbb A_1$.  After inverting $c$, the element $\partial=c^{-1}u$ satisfies $[\partial,c]=1$, and again $x\mapsto c$, $y\mapsto\partial$ gives a surjection from $\mathbb A_1[x^{-1}]$ to $Z[c^{-1}]$, which is an isomorphism because the source is simple.  In both cases $Z$ and $\mathbb A_1$ admit isomorphic denominator localizations, so $Z\sim_{\rm bir}\mathbb A_1$.
\end{proof}

\subsection{Affine type}

For an untwisted affine algebra of rank $r+1$, the space $\mathfrak h_0$ is one-dimensional.  Choose $z_0$ spanning $\mathfrak h_0$.  Corollary~\ref{cor:corank-separation} gives
$Z\cong k[z_0]\otimes Z_{\rm ess}$.
In affine type $Z_{\rm ess}$ has infinite Gelfand--Kirillov dimension.  Indeed, the positive imaginary subalgebra contains commuting elements $h\otimes t^m$ for all $m>0$, so $U(\mathfrak n_+)$ contains polynomial subalgebras of arbitrarily large finite rank and therefore has infinite Gelfand--Kirillov dimension.  Since $U(\mathfrak n_+)\subseteq L\cong Z\otimes\mathbb A_n[\mathbf x^{-1}]$, the algebra $L$ has infinite Gelfand--Kirillov dimension.  If $Z_{\rm ess}$ had finite Gelfand--Kirillov dimension, then $Z\cong k[z_0]\otimes Z_{\rm ess}$ and hence $L$ would have finite Gelfand--Kirillov dimension, a contradiction; compare \cite[Chapter~3]{KrauseLenagan2000}.  The examples below make the null-root unit and a finite generating seed for $Z_{\rm ess}$ explicit.

\subsubsection{Type $A_1^{(1)}$}

Use $\alpha_0=\delta-\alpha$, $\alpha_1=\alpha$ and the order $(0,1)$.  Then $\beta_0=\delta-\alpha$ and $\beta_1=2\delta-\alpha$, and in the loop realization of $\widehat{\mathfrak{sl}}_2$ one may take $X_0=ft$ and $X_1=ft^2$.
If $p_0,p_1\in\mathfrak h$ satisfy $\alpha_i(p_j)=\delta_{ij}$, then $b_0=p_0-2p_1$ and $b_1=p_1$.
Moreover $\delta=\beta_1-\beta_0$ and $X^\delta=X_1X_0^{-1}$.
Since $e_0=X_0$, its normalized coefficient is $1$.  For the normalized coefficient $a=e_1X_0^2X_1^{-1}$, the Coxeter window is $\{00,10,20,01,11,02\}$.
With the loop normalization above, the coefficient at $20$ is the constant $-2$.  Thus the five remaining projected coefficients generate $Z_{\rm ess}$.  In particular, although the null-root lattice direction is already contained in the Weyl factor, coefficients coming from the imaginary-root layers remain visible in the residual core.

\subsubsection{Type $A_2^{(1)}$}

Take the order $(0,1,2)$ and the standard loop realization.  The Coxeter roots and root vectors are
\[
\beta_0=\delta+\varepsilon_3-\varepsilon_1,
\quad
\beta_1=\delta+\varepsilon_3-\varepsilon_2,
\quad
\beta_2=2\delta+\varepsilon_3-\varepsilon_1,
\]
$X_0=E_{31}t,\qquad X_1=E_{32}t,\qquad X_2=E_{31}t^2$.
The dual coordinates are $b_0=p_0-p_1-p_2$, $b_1=p_1-p_2$, and $b_2=p_2$, while $\delta=\beta_2-\beta_0$ and $X^\delta=X_2X_0^{-1}$.
The normalized coefficient of $e_1$ has window $\{000,001,100\}$, while the normalized coefficient of $e_2$ has window $\{000,100,010,001,110,011,020\}$.  The position $100$ in the first window and the position $110$ in the second are constant.  The coefficient at $001$ in the first window and the coefficient at $020$ in the second are nonzero multiples of the same normalized element, both coming from $E_{32}t^2$.  After discarding the two constant positions and identifying this duplicate, seven projected coefficients remain and generate $Z_{\rm ess}$.  By Proposition~\ref{prop:window-generation}, this finite seed generates the contributions of the higher imaginary-root layers to $Z_{\rm ess}$.

\subsubsection{Type $B_2^{(1)}$}

For the rank-two non-simply-laced affine diagram, which is $C_2^{(1)}$ in the notation of \cite{Kac1990}, we use the orthogonal labeling
\[
\alpha_0=\delta-\varepsilon_1-\varepsilon_2,
\qquad
\alpha_1=\varepsilon_1-\varepsilon_2,
\qquad
\alpha_2=\varepsilon_2,
\]
with the order $(0,1,2)$.  Thus $0$ and $1$ are the long disconnected nodes and $2$ is short.  The Coxeter roots are
\[
\beta_0=\delta-\varepsilon_1-\varepsilon_2,
\qquad
\beta_1=\varepsilon_1-\varepsilon_2,
\qquad
\beta_2=\delta-\varepsilon_2.
\]
If $p_0,p_1,p_2\in\mathfrak h$ satisfy $\alpha_i(p_j)=\delta_{ij}$, then $b_0=p_0-p_2$, $b_1=p_1-p_2$, and $b_2=p_2$. Moreover $\delta=-\beta_0-\beta_1+2\beta_2$.
One may take
\[
X_0=e_{-(\varepsilon_1+\varepsilon_2)}t,
\qquad
X_1=e_{\varepsilon_1-\varepsilon_2},
\qquad
X_2=e_{-\varepsilon_2}t,
\]
and therefore $X^\delta=X_0^{-1}X_1^{-1}X_2^2$.
Since $e_0=X_0$ and $e_1=X_1$, their normalized coefficients are $1$.  The nonzero window positions of the normalized coefficient of $e_2$ are $000,100,010,001$, $110,101,011,002$.
The coefficient at $110$ is constant, while the numerator occurring at $001$ lies in the root space $\mathfrak g_{\delta}$.  Hence the seven remaining projected coefficients generate $Z_{\rm ess}$.  By Proposition~\ref{prop:window-generation}, the same finite seed generates the contributions of the higher imaginary-root layers to $Z_{\rm ess}$.

\begin{remark}
These affine examples have the same structural conclusion.  The homogeneous null-root unit $X^\delta$ lies in the explicit Weyl factor, but normalized coefficients attached to imaginary-root spaces contribute to $Z_{\rm ess}$.  Thus the resulting birational equivalence separates the lattice direction $\delta$ while retaining these residual imaginary-root contributions in the birational model.
\end{remark}

\bibliographystyle{amsalpha}
\bibliography{references}

\end{document}